\documentclass[reqno, 12pt]{amsart}
\makeatletter
\let\origsection=\section \def\section{\@ifstar{\origsection*}{\mysection}} 
\def\mysection{\@startsection{section}{1}\z@{.7\linespacing\@plus\linespacing}{.5\linespacing}{\normalfont\scshape\centering\S}}
\makeatother        
\usepackage{amsmath,amssymb,amsthm}    
\usepackage{mathabx}\changenotsign    
\usepackage{mathrsfs} 
\usepackage{dsfont}

\usepackage{xcolor}  	
\usepackage[backref]{hyperref}
\hypersetup{
	colorlinks,
    linkcolor={red!60!black},
    citecolor={green!60!black},
    urlcolor={blue!60!black}
}

\usepackage[open,openlevel=2,atend]{bookmark}

\usepackage[abbrev,msc-links,backrefs]{amsrefs} 
\usepackage{doi}

\renewcommand{\PrintDOI}[1]{\doi{#1}}

\usepackage[T1]{fontenc}
\usepackage{lmodern}
\usepackage[babel]{microtype}

\usepackage[english]{babel}

\usepackage{geometry}
\usepackage{enumitem}

\let\polishlcross=\l
\def\l{\ifmmode\ell\else\polishlcross\fi}

\mathcode`l="8000
\begingroup
\makeatletter
\lccode`\~=`\l
\DeclareMathSymbol{\lsb@l}{\mathalpha}{letters}{`l}
\lowercase{\gdef~{\ifnum\the\mathgroup=\m@ne \ell \else \lsb@l \fi}}\endgroup

\def\qand{\quad\text{and}\quad}
\def\qqand{\qquad\text{and}\qquad}

\let\emptyset=\varnothing
\let\setminus=\smallsetminus

\makeatletter
\def\moverlay{\mathpalette\mov@rlay}
\def\mov@rlay#1#2{\leavevmode\vtop{   \baselineskip\z@skip \lineskiplimit-\maxdimen
   \ialign{\hfil$\m@th#1##$\hfil\cr#2\crcr}}}
\newcommand{\charfusion}[3][\mathord]{
    #1{\ifx#1\mathop\vphantom{#2}\fi
        \mathpalette\mov@rlay{#2\cr#3}
      }
    \ifx#1\mathop\expandafter\displaylimits\fi}
\makeatother

\newcommand{\PP}{\mathds{P}}

\renewcommand*{\Pr}{\PP}
\newcommand*{\E}{\EE}

\newtheorem{theorem}{Theorem}
\newtheorem{lemma}[theorem]{Lemma}

\newtheorem{prop}[theorem]{Proposition}

\newtheorem{fact}[theorem]{Fact}

\newtheorem{claim}[theorem]{Claim}
\newtheorem{cor}[theorem]{Corollary}

\let\theta=\vartheta
\let\rho=\varrho
\let\phi=\varphi

\def\tF{\widetilde{F}}

\def\NN{\mathds N}

\def\PP{\mathds P}
\def\EE{\mathds E}

\def\cQ{{\mathcal Q}}

\def\bchi{\bar\chi}
\def\qI{q_{\rm I}}
\def\qII{q_{\rm II}}

\DeclareMathOperator{\aut}{aut}

\begin{document}

\title{Ramsey properties of random graphs and Folkman numbers}

\author{Vojt\v{e}ch R\"{o}dl}
\address{Department of Mathematics and Computer Science, 
Emory University, Atlanta, USA}
\email{rodl@mathcs.emory.edu}

\author{Andrzej Ruci\'nski}
\address{A. Mickiewicz University, Department of Discrete Mathematics, Pozna\'n, Poland}
\email{rucinski@amu.edu.pl}

\author{Mathias Schacht}
\address{Fachbereich Mathematik, Universit\"at Hamburg, Hamburg, Germany}
\email{schacht@math.uni-hamburg.de}

\thanks{V.~R\"odl was supported  by NSF grants DM 080070. 
	A.~Ruci\'nski was supported by the Polish NSC grant~2014/15/B/ST1/01688 
	and parts of the research 
	were performed during visits at Emory University (Atlanta).
	M.~Schacht was supported through the \emph{Heisenberg-Programme} of the DFG\@.}


\begin{abstract} For two graphs, $G$ and $F$, and an integer $r\ge2$ we write $G\rightarrow (F)_r$ if every $r$-coloring of the edges of $G$
results in a monochromatic copy of $F$. In 1995, the first two authors  established a threshold edge
probability for the Ramsey property $G(n,p)\to (F)_r$, where $G(n,p)$ is a random graph
obtained by including each edge of the complete graph on $n$ vertices, independently, with
probability $p$.  The original proof was based on the regularity lemma of
Szemer\'edi and this led to  tower-type dependencies between the involved parameters.  Here, for $r=2$, we provide  a self-contained proof of a
quantitative version of  the Ramsey threshold theorem with only double exponential dependencies between the
constants. As a corollary we obtain a double exponential upper bound on the 2-color Folkman numbers. By a different proof technique, a similar result was obtained independently
by Conlon and Gowers.
\end{abstract}

\maketitle

\section{Introduction}\label{intro}
For two graphs, $G$ and $F$, and an integer $r\ge2$ we write $G\rightarrow (F)_r$ if every $r$-coloring of the edges of $G$
results in a monochromatic copy of $F$. By a copy we mean here a subgraph of $G$ isomorphic to $F$.
Let $G(n,p)$ be the binomial random graph, where each of $\binom n2$ possible edges is present,
independently, with probability $p$.
In \cite{rr} the first two authors  established a threshold edge
probability for the Ramsey property $G(n,p)\to (F)_r$.

For a graph $F$,  let $v_F$ and $e_F$ stand for, respectively, the number of vertices and edges of~$F$. Assuming $e_F\ge1$, define
\begin{equation}\label{eq:dF}
d_{F}=
 \begin{cases}
   \frac{e_{F}-1}{v_{F}-2}&\text{if}\quad e_{F}>1\\
   \frac12&\text{if}\quad e_{F}=1
 \end{cases}\,,
\end{equation}
and
  \begin{equation}\label{eq:mF}
  m_{F}=\max\{d_{H}\colon H\subseteq F\ \text{and}\ e_{H}\geq
  1\}\,.
  \end{equation}
Let $\Delta(F)$ be the maximum vertex degree in $F$. Observe that $m_F=\tfrac12$ for every $F$ with $\Delta(F)=1$, while for every $F$ with
$\Delta(F)\ge2$ we have $m_F\ge1$. Moreover, for every $k$-vertex graph $F$,
\[
    m_F\leq m_{K_k}=\frac{k+1}{2}\,.
\]
We now state the main result of \cite{rr} in a slightly abridged form.

\begin{theorem}[\cite{rr}]
\label{95} For every integer $r\ge2$ and a graph $F$ with $\Delta(F)\ge2$ there exists a constant
$C_{F,r}$ such that if $p=p(n)\ge C_{F,r}n^{-1/m_F}$ then
$$\lim_{n\to\infty}\PP(G(n,p)\rightarrow (F)_r)=1.$$
\end{theorem}

\noindent The original proof of Theorem \ref{95} was based on the regularity lemma of
Szemer\'edi~\cite{Szem} and this led to  tower-type dependencies on the involved parameters. In
\cite{rrs} it was noticed that for \emph{two} colors the usage of the regularity lemma could be
replaced by a simple Ramsey-type argument. Here we follow that thread and for $r=2$ prove a
quantitative version of  Theorem \ref{95} with only double exponential dependencies between the
constants.

In order to state the result, we first define inductively four parameters indexed by the number of
edges of a $k$-vertex graph $F$. For fixed $k\geq 3$ we set
\begin{equation}\label{eq:L4-1}
    a_1=\frac12,\qquad b_1=\frac18, \qquad C_1=1,\qquad\text{and}\qquad n_1=1
\end{equation}
and for each $i=1,\dots,\binom{k}{2}-1$,  define
\begin{equation}\label{rec}
    a_{i+1}=\frac{a_i^{19k^4}}{2^{55k^6}},\quad
    b_{i+1}=\frac{a_i^{37k^2}}{2^{118k^4}}b_i^4,\quad
    C_{i+1}=\frac{2^{122k^4}}{b_i^{4}a_i^{37k^2}}C_i,\quad \text{and}\quad
    n_{i+1}=\frac{2^{14k^3}}{a_{i}^{4k}}n_{i}\,.
\end{equation}

Note that $a_i$ and $b_i$ decrease with $i$, while $C_i$ and $n_i$ increase.
Finally, for a graph $F$ on~$k$ vertices, denote by
$$
\mu_F=\binom nk\frac{k!}{\aut(F)}p^{e_F}$$
the expected number of  copies of $F$ in $G(n,p)$
and note that
 \begin{equation}\label{miu}
 \binom nk p^{e_F}\le\mu_F\le n^{k} p^{e_F}=n^{v_F} p^{e_F}.
 \end{equation}

For  a real number $\lambda>0$ we write $G\xrightarrow{\lambda}F$ if every 2-coloring of the edges of $G$ produces at least $\lambda$ monochromatic copies of $F$. We call a graph $F$ \emph{$k$-admissible} if $v_F=k$ and either $e_F=1$ or $\Delta(F)\ge2$.
Now, we are ready to state a quantitative version of Theorem~\ref{95}.

\begin{theorem}\label{rg}
For every $k\ge3$, every $k$-admissible graph $F$, and for all
$n\ge n_{e_F}$ and $p\geq C_{e_F}n^{-1/m_F}$ we have
$$\PP\left(G(n,p)\xrightarrow{a_{e_F}\mu_F}F\right)\ge 1-\exp(-b_{e_F}p\tbinom n2).$$
\end{theorem}
\noindent
Note that,  for $r=2$, Theorem \ref{95} is an immediate
corollary of Theorem \ref{rg}.

Another consequence of Theorem \ref{rg} concerns    Folkman numbers.  Given an integer $k\ge3$, the Folkman number $f(k)$ is the smallest integer $n$ for which there exists an $n$-vertex graph $G$ such that $G\to (K_k)_2$ but $G\not\supset K_{k+1}$.
In the special case of $F=K_k$ and $r=2$, Theorem \ref{rg}, with  $p=C_{\binom k2}n^{-\tfrac2{k+1}}$,
provides a lower bound on  $\PP(G(n,p)\rightarrow (K_k)_2)$. In Section \ref{proofcor}, by a standard application of the FKG inequality, we also estimate from below
$\PP(G(n,p)\not\supset K_{k+1})$, so that the sum of the two probabililities is  strictly greater than 1.  This, after a careful analysis of the involved constants, provides a self-contained
derivation  of a double exponential bound for $f(k)$.

\begin{cor}\label{main}
There exists an absolute constant $c>0$ such that for every $k\ge 3$
$$f(k)\le 2^{k^{ck^2}}.$$
\end{cor}

\noindent  Independently,
a similar double exponential bound  (with arbitrarily many colors) was obtained by Conlon and
Gowers \cite{CG}. The method used in \cite{CG} is quite different from ours and allows for  a further generalization to hypergraphs. 
After  Theorem \ref{rg} as well as the result  in \cite{CG} had been proved, we learned that
 Nenadov and Steger \cite{Steger}  have found a new  proof of Theorem \ref{95} by means of the celebrated containers' method. In \cite{folk}, we used the ideas from \cite{Steger} to obtain the bound $f(k)\le 2^{O(k^4\log k)}$ which, at least for large $k$,
supersedes Colorary~\ref{main}. However, the advantage of  our approach here is that the proofs of both Theorem \ref{rg} and Corollary  \ref{main}, as opposed to those in \cite{folk},
are self-contained and, in case of Theorem \ref{rg}, incorporate the original ideas from \cite{rr}.

The paper is organized as follows. In Section \ref{proofofmain} we prove our main result, Theorem~\ref{rg}. This is preceded by Section \ref{prem} collecting preliminary results needed in  the main body of the proof.
 Section \ref{proofcor} is devoted to  a   proof of Corollary \ref{main}.

\section{Preliminary results}\label{prem}
Before we start with  the proof of Theorem~\ref{rg}, we need to recall abridged versions of two useful facts from \cite{JLR}*{Lemmas~2.52 and~2.51}
(see also \cites{rr,rrs}), which we formulate as Propositions~\ref{del} and~\ref{upper} below.

Given a set $\Gamma$ and a real number $p$, $0\le p\le 1$, let $\Gamma_p$ be the random binomial
subset of~$\Gamma$, that is, a subset obtained by independently including each element of~$\Gamma$
with probability~$p$. Further, given an increasing family $\mathcal Q$ of subsets of a set $\Gamma$
and an integer $h$, we denote by~$\mathcal Q_h$ the subfamily of $\mathcal Q$ consisting of the
sets $A\in \mathcal Q$ having the property that  all  subsets of $A$ with at least $|A|-h$ elements
still belong to $\mathcal Q$.

\begin{prop}\label{del}
Let $0<c<1$, $\delta=c^2/9$,  $Np\ge 72/\delta^2=2^33^6/c^4$, and $h=\delta Np/2$.
Then for every increasing family $\mathcal Q$ of subsets of an $N$-element set $\Gamma$ the following holds.
If
$$
\PP\left(\Gamma_{(1-\delta)p}\not\in\mathcal Q\right)\le\exp(-cNp)
$$
then
$$
\PP\left(\Gamma_p\not\in\mathcal{Q}_h\right)\le\exp(-\delta^2Np/9)\,.
$$
\end{prop}

\begin{proof}
We shall apply~\cite{JLR}*{Lemma~2.52}, which is very similar to Proposition~\ref{del}. Lemma~2.52 from~\cite{JLR} states
that if $c$ and $\delta>0$ satisfy
\begin{equation}\label{eq:2.52-1}
    \delta(3+\log(1/\delta))\leq c
\end{equation}
and
\[
    \PP(\Gamma_{(1-\delta)p}\not\in\cQ)\le\exp(-cNp)
\]
then
\begin{equation}\label{eq:2.52-3}
    \PP(\Gamma_p\not\in\cQ_h)\le3\sqrt{Np}\exp(-cNp/2)+\exp(-\delta^2Np/8)\,.
\end{equation}
To this end we first note that by assumption of Proposition~\ref{del} we have $\delta <1/9$. Since
$\sqrt{x}(\log(1/x)$ is increasing for $x\in(0,1/\textrm{e}^2]$ it follows for every $\delta\le1/9$
that
\[
    \sqrt{\delta}\log(1/\delta)\leq\frac{\log(9)}{3} \leq 2\,.
\]
Consequently, $\sqrt{\delta}(3+\log(1/\delta))\leq 3$ and owing to the assumption $\delta=c^2/9$
this is equivalent to~\eqref{eq:2.52-1}.
Moreover,  since $Np\geq 2^33^6/c^4>(12/c)^2$ we have
$$
3\sqrt{Np}\le \exp({3\sqrt{Np}})\le \exp({cNp/4})\,.
$$
Hence,~\eqref{eq:2.52-3} yields
\begin{align*}
   \PP(\Gamma_p\not\in\cQ_h)&\le \exp(-cNp/4)+\exp(-\delta^2Np/8)
    \le
    2\exp(-\delta^2Np/8)\\
    &\le
    \exp(-\delta^2Np/8+1)\le \exp(-\delta^2Np/9),
\end{align*}
where the last inequality follows by our  assumption $Np\geq 72/\delta^2$. 
\end{proof}

The following result has appeared in \cite{JLR} as Lemma~2.51.  We state it here for $t=2$ only.
\begin{prop}[{\cite{JLR}}]\label{upper}
Let $\mathcal S\subseteq\binom\Gamma s$, $0\le p\le 1$, and $\lambda=|\mathcal S|p^s$. Then for every nonnegative integer $h$, with probability at least $1-\exp(-\tfrac{h}{2s})$, there exists a subset $E_0\subseteq \Gamma_p$ of size $h$ such that $\Gamma_p\setminus E_0$ contains at most $2\lambda$ sets from~$\mathcal S$.
\end{prop}

In the proof of Theorem \ref{rg} we will also use an elementary fact about $(\rho,d)$-dense graphs. For constants $\rho$ and $d$  with $0<d,\rho\le 1$ we call an
$n$-vertex  graph $\Gamma$ $(\rho,d)$-dense if every induced subgraph on  $m\ge \rho n$ vertices
contains at least $d(m^2/2)$ edges. It follows by an easy averaging argument that it suffices to
check the above inequality only for $m=\lceil \rho n\rceil$. Note also that every induced subgraph
of a $(\rho,d)$-dense $n$-vertex graph on at least $cn$ vertices is $(\tfrac{\rho}c,d)$-dense.

It turns out that for a suitable choice of the parameters, $(\rho,d)$-dense graphs enjoy a
Ramsey-like property. For a two-coloring of (the edges of) $\Gamma$ we call a sequence of vertices
$(v_1,\dots,v_\l)$ \emph{canonical} if for each $i=1,\dots,\l-1$ all the edges $\{v_i,v_j\}$, for
$j>i$ are of the same color.

\begin{prop}\label{rhodee} For every $\l\ge 2$ and $d\in(0,1)$, if $n\ge 2(4/d)^{\l-2}$ and $0<\rho\le (d/4)^{\l-2}/2$, then  every two-colored $n$-vertex
$(\rho,d)$-dense graph $\Gamma$ contains at least
\[
    f_n(\l):=\left(\frac14\right)^{\binom{\l+1}2}d^{\binom \l2} n^\l
\]
canonical sequences of length $\l$.
\end{prop}

\begin{proof}
First, note that as long as $\rho\le 1/2$ every $(\rho,d)$-dense graph contains at least $n/2$
vertices with degrees at least $dn/2$. Indeed, otherwise a set of $m=\lceil (n+1)/2\rceil$ vertices
of degrees smaller than $dn/2$ would induce less than $mdn/4\le d(m^2/2)$ edges, a contradiction.

We prove Proposition \ref{rhodee} by induction on $\l$. For $\l=2$, every ordered pair of adjacent vertices is a canonical sequence and there are at least $2d\binom n2>f_n(2)$ such pairs if $n\ge2$. Assume that
the proposition is true for some $\l\ge2$ and consider an $n$-vertex $(\rho,d)$-dense graph $\Gamma$, where $\rho\le (d/4)^{\l-1}/2$ and $n\ge 2(4/d)^{\l-1}$. As observed above, there is a set $U$ of at least $n/2$ vertices with degrees at least $dn/2$. Fix one vertex $u\in U$ and let $M_u$ be a set of at least $dn/4$ neighbors of $u$ connected to $u$ by edges of the same color. Let $\Gamma_u=\Gamma[M_u]$ be the subgraph of $\Gamma$ induced  by the set $M_u$. Note that $\Gamma_u$ has $n_u\ge dn/4\ge 2(4/d)^{\l-2}$ vertices and  is $(\rho_u,d)$-dense with $\rho_u\le (d/4)^{\l-2}/2$. Hence, by the induction assumption, there are at least
$$
f_{n_u}(\l)
\ge
\left(\frac14\right)^{\binom{\l+1}2}d^{\binom \l2} \left(\frac{dn}{4}\right)^\l
=
\left(\frac14\right)^{\binom{\l}2+\l}d^{\binom {\l+1}2}n^\l
$$
canonical sequences of length $\l$ in $\Gamma_u$. Each of these sequences preceded by the vertex~$u$ makes a canonical sequence of length $\l+1$ in $\Gamma$. As there are at least $n/2$ vertices in $U$, there are at least
$$
\frac n2f_{n_u}(\l)
\ge
\left(\frac14\right)^{\binom{\l+2}2}d^{\binom {\l+1}2}n^{\l+1}$$
canonical sequences of length $\l+1$ in $\Gamma$. This completes the inductive proof of Proposition~\ref{rhodee}. 
\end{proof}

\begin{cor}\label{mon}
 For every $k\ge2$, every graph $F$ on $k$ vertices, and every $d\in(0,1)$, if $n\ge (4/d)^{2k}$
 and $0<\rho\le(d/4)^{2k}$, then  every two-colored
 $n$-vertex, $(\rho,d)$-dense graph $\Gamma$ contains at least
    $\gamma n^k$
 monochromatic  copies of $F$, where $\gamma=d^{2k^2}2^{-5k^2}.$
\end{cor}

\begin{proof} Every canonical sequence $(v_1,\dots, v_{2k-2})$  contains a monochromatic  copy of $K_k$.
Indeed, among the vertices $v_1,\dots, v_{2k-3}$, some  $k-1$ have the same color on all the
``forward'' edges. Therefore, these vertices together with vertex $v_{2k-2}$ form a monochromatic
copy of $K_k$. On the other hand, every such copy is contained in no more than
\[
	k!\binom{2k-2}kn^{k-2}=(2k-2)_kn^{k-2}
\] 
canonical sequences of length $2k-2$. Finally, every copy
of $K_k$ contains at least one copy of $F$, and different copies of $K_k$ contain different copies
of $F$. Consequently, by Proposition \ref{rhodee}, every two-colored
 $n$-vertex, $(\rho,d)$-dense graph $\Gamma$ contains at least
 \[
    \frac{f_n(2k-2)}{(2k-2)_kn^{k-2}}=\frac1{(2k-2)_k}\left(\frac14\right)^{\binom{2k-1}2}d^{\binom{2k-2}2} n^k
    >
    \frac{d^{2k^2}}{2^{5k^2}}n^k
 \]
 monochromatic  copies of $F$. 
 \end{proof}

\section{Proof of Theorem \ref{rg}}\label{proofofmain}

\subsection{Preparations and outline}\label{prepout}
For given $n\in\NN$, $p\in(0,1)$, and a $k$-vertex graph $F$ we denote by $X_F$ the random variable counting the number of copies of $F$ in $G(n,p)$.
We also recall that $\mu_F=\EE X_F$.

For fixed $k\geq 3$ we prove Theorem~\ref{rg} by induction on $e_F$.
We may assume $n\ge k$, as for~$n<k$ we have $\mu_F=0$ and there is nothing to prove.

\paragraph*{Base case.}
Let $F_1$ be a graph consisting of one edge and $k-2$ isolated vertices. Note 
that~$m_{F_1}=1/2$
(see~\eqref{eq:mF}) and for every two-coloring of the edges of $G(n,p)$ every copy of~$F_1$ in
$G(n,p)$ is monochromatic. Clearly,
$$X_{F_1}=\binom{n-2}{k-2}X_{K_2}\quad\mbox{ and }\quad\mu_{F_1}=\binom{n-2}{k-2}\mu_{K_2}= \binom{n-2}{k-2}\binom n2 p.$$
 Thus,  by Chernoff's bound (see, e.g.,~\cite{JLR}*{ineq.~(2.6)}) we have
$$
    \PP\left(X_{F_1}\le \frac12\mu_{F_1}\right)=\PP\left(X_{K_2}\le \frac12\mu_{K_2}\right)
    \le \exp\left(-\frac18\binom n2p\right)\,,
$$
which holds for any values of $p$ and $n$.
Hence, Theorem~\ref{rg} follows for $F=F_1$ and
with the  constants $a_{1}=1/2$, $b_{1}=1/8$, and $C_{1}=n_{1}=1$
as given in~\eqref{eq:L4-1}.

\paragraph*{Inductive step.}
Given a graph $G$, an edge $f$ of $G$ and a nonedge $e$, that is an edge of the complement of $G$, we denote by $G-f$ a graph obtained from $G$ by removing $f$, and by~$G+f$ a graph obtained by adding $e$ to $G$.
Let $F_{i+1}$ be a graph with $i+1\geq 2$ edges and maximum degree $\Delta(F_{i+1})\geq 2$. If $i+1\geq 3$, then we can remove one edge
from $F_{i+1}$ in such a way that the resulting graph $F_i$ still contains at least one vertex of degree at least two, i.e., $\Delta(F_{i})\geq 2$.
If $i+1=2$, the graph $F_{i+1}=F_2$ consists of a path of length two and $k-3$ isolated vertices and removing any of the two edges results in the graph~$F_i=F_1$.
In either case, we may fix an edge $f\in E(F_{i+1})$ such that the graph  $F_i=F_{i+1}-f$ is $k$-admissible.
Hence, we can assume that Theorem~\ref{rg} holds for $F_i$ and for the constants $a_i$, $b_i$, $C_i$, and~$n_i$ inductively defined
by~\eqref{eq:L4-1} and~\eqref{rec}.

We have to show that Theorem~\ref{rg} holds for $F_{i+1}$ and constants $a_{i+1}$, $b_{i+1}$,
$C_{i+1}$, and~$n_{i+1}$ given in~\eqref{rec}. To this end, let $n\geq n_{i+1}$ and $p\geq
C_{i+1}n^{-1/m_{F_{i+1}}}$. We will expose the random graph $G(n,p)$ in two independent rounds
$G(n,p_{\rm I})$ and $G(n,p_{\rm II})$ and have 
\[
	G(n,p)=G(n,p_{\rm I})\cup G(n,p_{\rm II})\,.
\] 
For
that, we will fix $p_{\rm I}$ and $p_{\rm II}$ as follows. First we fix auxiliary
constants\footnote{The proof requires several auxiliary constants which at first may appear a bit
unmotivated. For example, we  now define $\delta_{\rm II}$, while $\delta_{\rm I}$ is to be defined
only later. Both $\delta$'s will be used in  applications of Proposition~\ref{del}.}
\begin{equation}\label{eq:dgd}
    d=\frac{a_i^2}{64^{k^2}}\,,\qquad
    \rho=\left(\frac{d}{4}\right)^{2k}\,,\qquad
    \gamma=\frac{d^{2k^2}}{2^{5k^2}}\,,
    \qquad
    \delta_{\rm II}=\frac{\gamma^4}{9\cdot 16^{k^2}}\,,
    \qand
    \alpha=\frac{\delta_{\rm II}^2\gamma}{36}\,.
\end{equation}
Then $p_{\rm I}$ and $p_{\rm II}\in(0,1)$ are defined by the equations
\begin{equation}\label{eq:p1p2}
    p=p_{\rm I}+p_{\rm II}-p_{\rm I}p_{\rm II}
    \qqand
    p_{\rm I}=\alpha p_{\rm II}\,.
\end{equation}
Clearly, we have
\begin{equation}\label{eq:ps}
    p\ge p_{\rm II}\ge\frac{p}{2}
    \ge\alpha p\ge\alpha p_{\rm II}=
    p_{\rm I}\ge\alpha\frac{p}{2}\,.
\end{equation}
We continue with a short outline of the main ideas of the forthcoming proof.

\subsubsection*{Outline.} At first we consider a two-coloring $\chi$, with colors red and blue,
of the edges of~$G(n,p_{\rm I})$ (the so-called first round). Owing to the induction assumption (Theorem~\ref{rg} for~$F_{i}$) we
note that with high probability the coloring $\chi$ yields many monochromatic copies of~$F_i$. We
will say that an unordered pair of vertices $e=\{u,v\}$ is \emph{$\chi$-rich}  if $G(n,p_{\rm
I})+e$ possesses ``many'' (to be defined later) copies of $F_{i+1}$, in which $e$ plays the role of
the edge~$f$ and the rest is a monochromatic copy of $F_i$. Let $\Gamma_\chi$ be an auxiliary graph
of all $\chi$-rich pairs. We will show that with `high' probability  (to be specified later),
$\Gamma_\chi$ is, in fact, $(\rho,d)$-dense for $d$ and $\rho$ as in~\eqref{eq:dgd} (Claim \ref{Gamma}).

 To this end, note that if the  monochromatic copies of~$F_i$  were clustered at relatively few
pairs, then we might fall short of proving Claim \ref{Gamma}. However,
we will show that in the random graph $G(n,p_{\rm I})$ it is unlikely that many copies of $F_i$ share the same 
pair of vertices. For that, we will consider the distribution of the graphs $T$ consisting of two copies of
$F_i$ which share the vertices of a missing edge $f$ (and possibly other vertices). We will show
that the number of those copies is of the same order of magnitude as its expectation (Fact~\ref{T}), and will also require that this
holds with high probability. Such a sharp concentration result
is known to be false, but Proposition~\ref{upper} asserts that it can be obtained on the cost of
removing a few edges of $G(n,p_{\rm I})$.

The auxiliary graph $\Gamma_\chi$ is naturally two-colored (by azure and pink), since every $\chi$-rich pair closes either
many blue or many red copies of $F_i$ (or both and then we pick the color for that edge, azure or blue,
arbitrarily). Consequently, Corollary \ref{mon} yields many monochromatic copies of $F_{i+1}$
in $\Gamma_\chi$ and at least half of them are colored, say, pink. That is, there are many copies of
$F_{i+1}$ in $\Gamma_\chi$ such that each of their edges closes many red copies of $F_i$ in
$G(n,p_{\rm I})$ under the coloring $\chi$. By Janson's inequality combined with Proposition~\ref{del}, with high probability, many pink copies will be still present in $\Gamma_{\chi}\cap G(n,p_{\rm II})$
 (second round) even after a fraction of edges is deleted. Thus, we are facing a `win-win' scenario. Namely, if an extension of $\chi$
 colors only few pink edges of $\Gamma_{\chi}\cap G(n,p_{\rm II})$ red then, by the above, many copies of $F_{i+1}$ in $\Gamma_{\chi}\cap G(n,p_{\rm II})$ have to be
colored completely blue.
 Otherwise, many pink edges of $\Gamma_{\chi}\cap G(n,p_{\rm II})$ are red, which, by the definition of a pink edge, results in many red copies of $F_{i+1}$ in $G(n,p)$.
 
\paragraph*{Useful estimates.} For the verification of several inequalities in the proof, it will
be useful to  appeal to the following lower bounds for $\gamma$, $\alpha$, and $\rho$ in terms of
powers of $a_i$ and $2$. From the definitions in~\eqref{eq:dgd}, for sufficiently large $k$, one
obtains the following bounds.
\begin{equation}\label{eq:dgdr}
\begin{split}
    \gamma
    =&\frac{a_i^{4k^2}}{2^{12k^4+5k^2}}
    \ge\frac{a_i^{4k^2}}{2^{13k^4}}\,,\\
    \alpha
    =&
    \frac{a_i^{36k^2}}{3^6\cdot 2^{108k^4+53k^2+2}}
    \ge
    \frac{a_i^{36k^2}}{2^{109k^4}}\,,\\
    \rho
    =&
    \frac{a_i^{4k}}{2^{12k^3+4k}}
    \ge
    \frac{a_i^{4k}}{2^{13k^3}}\,.
    \end{split}
\end{equation}

\medskip

\noindent We will also make use of the inequalities
\begin{equation}\label{npC}
np\ge C_{i+1},
\end{equation}
valid because $m_{F_{i+1}}\ge1$, and,  for every subgraph $H$ of $F_{i+1}$ with $v_H\ge3$,
\begin{equation}\label{npCH}
n^{v_H-2}p^{e_H-1}\ge C_{i+1}^{e_H-1},
\end{equation}
 valid because
$$m_{F_{i+1}}\ge d_H=\frac{e_H-1}{v_H-2}.$$
Of course, \eqref{npC} follows from \eqref{npCH}, by taking $H$ with $d_H=m_{F_{i+1}}$.

\subsection{Details of the proof}\label{detail}
\subsubsection*{First round.}
As outlined above, in the first round we want to show that with high probability
the random graph $G(n,p_{\rm I})$ has the property that for every two-coloring $\chi$ the auxiliary graph
$\Gamma_\chi$ (defined below) is $(\rho,d)$-dense. For that we set
\begin{equation}\label{eq:delta_I}
    \delta_{\rm I}=\frac{b_{i}^2}{36}
\end{equation}
and for a two-coloring $\chi$ call a pair $\{u,v\}$ of vertices  \emph{$\chi$-rich}
if it closes at least
\begin{equation}\label{eq:ell}
    \l=\frac{a_i}{4^{k^2}}(\rho n)^{k-2}p_{\rm I}^i
\end{equation}
monochromatic copies of $F_i$ in $G(n,p_{\rm I})$ to a copy of $F_{i+1}$. Then $\Gamma_\chi$
is an auxiliary $n$-vertex graph with the edge set being the set of \emph{$\chi$-rich} pairs.

Let $\mathcal{E}$ be the event (defined on $G(n,p_{\rm I})$) that  for every  two-coloring $\chi$
of $G(n,p_{\rm I})$ the graph $\Gamma_\chi$ is $(\rho,d)$-dense.

\begin{claim}\label{Gamma}
$$
 \PP(\mathcal{E})\ge1-\exp\left(-\frac{\delta_{\rm I}^2}{16^{k^2}}\binom{\rho n}{2}p_{\rm I}+n+2k^2\right)
$$
\end{claim}

Before giving the proof of Claim \ref{Gamma} we need one more fact.
 Let $\{T_1,T_2,\dots,T_t\}$ be the family of all pairwise
non-isomorphic graphs which are unions of two copies of $F_i$, say $F_i'\cup F_i''$, with the property that adding a single edge completes both, $F_i'$ and $F_i''$
to a copy of $F_{i+1}$. We will refer to these graphs  as \emph{double creatures} (of $F_i$).
Clearly, with some foresight of future applications,
\begin{equation}\label{eq:t}
    t\leq 2^{\binom{2k-2}{2}}\le2^{2k^2-4k}\le\frac{2^{2k^2-1}}{4\binom k2}\,.
\end{equation}
Let $X_j$ be the number of copies of $T_j$ in $G(U,p_{\rm I})$, $j=1,\dots,t$.
\begin{fact}\label{T}
For every $j=1,\dots,t$
$$\EE X_j\le(\rho n)^{2k-2}p_{\rm I}^{2i}\,.$$
\end{fact}
\begin{proof}
Let $T:=T_j=F_i'\cup F_i''$ be a double creature and set $S=F'_i\cap F_i''$. Then the expected number of
copies of $T$ is bounded from above  by
\[
    \EE X_T\overset{\eqref{miu}}{\leq} (\rho n)^{v_T}p_{\rm I}^{e_T}=\frac{(\rho n)^{2k}p_{\rm I}^{2i}}
    {(\rho n)^{v_S}p_{\rm I}^{e_S}}\,,
\]
and it remains to show that
$$
(\rho n)^{v_S}p_{\rm I}^{e_{S}}\ge (\rho n)^{2}.
$$

 There is nothing to prove when $v_S=2$ (and thus $e_S=0$).
 Otherwise, pick a pair of vertices $f$ in $T$ such that both,
$F_i'+f$ and $F_i''+f$, are isomorphic to $F_{i+1}$. Then $J:=S+f\subseteq F_{i+1}$. Note that
$e_J=e_S+1$ and  $3\le v_J=v_S\le k$. Since $C_{i+1}\ge2/\alpha$,
\begin{align*}
    (\rho n)^{v_S}p_{\rm I}^{e_{S}}
    &\overset{\eqref{eq:ps}}{\ge}
    (\rho n)^{v_J}\left(\frac{\alpha}{2}\right)^{e_S}p^{e_{J}-1}
    \overset{\eqref{npCH}}{\geq}
    \rho^{v_S-2}\left(\frac{\alpha}{2}\right)^{e_S}C_{i+1}^{e_{S}}(\rho n)^2\\
    &\overset{\phantom{\eqref{eq:ps}}}{\ge}\rho^k\frac\alpha2 C_{i+1}(\rho n)^2
    \overset{\eqref{eq:dgdr}}{\geq}
    \frac12\frac{a_i^{4k^2}}{2^{13k^4}}
    \frac{a_i^{36k^2}}{2^{109k^4}} C_{i+1}(\rho n)^2
    \overset{\eqref{rec}}{\geq}
    \frac{2^{13k^4-1}}{b_i^{4}}C_i(\rho n)^2\ge (\rho n)^2
    \,.\qedhere
\end{align*}
\end{proof}

\begin{proof}[Proof of Claim~\ref{Gamma}]
 Let $\chi$ be a two-coloring of $G(n,p_{\rm I})$. Fix a set
$U\subseteq [n]$ with  $|U|=\rho n$ (throughout we assume that $\rho n$ is an integer) and consider
the random graph $G(n,p_{\rm I})$ induced on~$U$
$$
    G(U,p_{\rm I}):=G(n,p_{\rm I})[U]\,.
$$
By the induction assumption, if $\rho n\ge n_i$ and $p_i\ge C_i(\rho n)^{-1/m_{F_i}}$ then, with high probability, there are many monochromatic copies of $F_i$ in $G(U,p_{\rm I})$.
For technical reasons that will become clear only later, we want to strengthen the above Ramsey property so that it is resilient to deletion of a small fraction of edges. For that we apply the induction assumption to the random
graph $G(U,(1-\delta_{\rm I})p_{\rm I})$, followed by an application of Proposition~\ref{del}.
We begin by verifying the assumptions of
Theorem~\ref{rg} with respect to $F_i$ and~$G(U,(1-\delta_{\rm I})p_{\rm I})$.
First,  note that
\begin{equation}\label{eq:rhon}
 |U|=\rho n\geq \rho n_{i+1}
 \overset{\eqref{eq:dgdr}}{\geq}
 \frac{a_i^{4k}}{2^{13k^3}}n_{i+1}
 \overset{\eqref{rec}}{=}
 \frac{a_i^{4k}}{2^{13k^3}}\cdot\frac{2^{14k^3}}{a_{i}^{4k}}n_{i}
 =2^{k^3}n_i
 \ge n_i\,.
\end{equation}
It remains to check that
\begin{equation}\label{eq:p1check}
    (1-\delta_{\rm I})p_{\rm I}\geq C_i(\rho n)^{-1/m_{F_i}}\,.
\end{equation}
To this end, we simply note that using  $\delta_{\rm I}\leq 1/2$, $\rho\le1$, and $m_{F_{i+1}}\geq \max\left(1, m_{F_{i}}\right)$
we have
\[
    (1-\delta_{\rm I})p_{\rm I}
    \overset{\eqref{eq:ps}}{\ge}
    \frac{\alpha p}{4}
    \geq
    \frac{\alpha}{4}C_{i+1}\rho^{1/m_{F_{i+1}}}(\rho n)^{-1/m_{F_{i+1}}}
    \geq
    \frac{\alpha}{4}C_{i+1}\rho(\rho n)n^{-1/m_{F_i}}.
\]

Furthermore, we have
\begin{equation}\label{eq:alpharho2C}
    \frac{\alpha\rho}{4}C_{i+1}
    \overset{\eqref{eq:dgdr}}{\geq}
    \frac{a_i^{36k^2+4k}}{2^{109k^4+13k^3+2}}\cdot C_{i+1}
    \overset{\eqref{rec}}{=}
    \frac{a_i^{37k^2}}{2^{110k^4}}\cdot \frac{2^{122k^4}C_i}{b_i^4a_i^{37k^2}}
    =
    \frac{2^{12k^2}C_i}{b_i^4}
    \ge C_i\,.
\end{equation}
and~\eqref{eq:p1check} follows. 

Thus, we are in position to apply the induction assumption to~$G(U,(1-\delta_{\rm I})p_{\rm I})$ and  $F_i$. Let
\begin{equation}\label{eq:mui}
    \mu:=\mu^{\rho,\delta_{\rm I}}_{F_i}
    :=
    \binom{\rho n}{k}\frac{k!}{\aut(F_i)}((1-\delta_{\rm I})p_{\rm I})^{i}
    \geq
    \frac{1}{4^{k^2}}(\rho n)^{k}p_{\rm I}^i
\end{equation}
denote the expected number of copies of $F_i$ in $G(U,(1-\delta_{\rm I})p_{\rm I})$. By Theorem \ref{rg} we infer that
\begin{align}
\PP\left(G(U,(1-\delta_{\rm I})p_{\rm I})\xrightarrow{a_i\mu}F_i\right)
&\ge
1-\exp\left(-b_i(1-\delta_{\rm I})p_{\rm I}\tbinom{\rho n}2\right)\nonumber\\
&\ge
1-\exp\left(-\tfrac{b_i}{2}p_{\rm I}\tbinom{\rho n}2\right)\,.\label{eq:R1del}
\end{align}
Next we head for an application of Proposition~\ref{del} with $c=b_i/2$, $\delta=\delta_{\rm I}$, $N=\binom{\rho n}{2}$, and~$p_{\rm I}$.
Note that, indeed, $\delta_{\rm I}=b_i^2/36= c^2/9$ (see~\eqref{eq:delta_I}). Moreover, using $\rho n\geq 3$ (see~\eqref{eq:rhon})
and (\ref{npC}) we see that
\[
p_{\rm I}\binom{\rho n}{2}
\overset{\eqref{eq:ps}}{\geq}
\frac{\alpha p}{2}\cdot \rho n
\geq
\frac{\alpha\rho}{2}\cdot C_{i+1}
\overset{\eqref{eq:alpharho2C}}{\geq}
\frac{2^{12k^3+1}}{b_i^4}
\ge
\frac{72}{\delta_{\rm I}^2}
\]
and the assumptions of Proposition~\ref{del} are verified. From~\eqref{eq:R1del} we infer by Proposition~\ref{del}
that with probability
at least
\begin{equation}\label{eq:q1a}
    1-\exp\left(-\frac{\delta_{\rm I}^2}9\binom{\rho n}{2}p_{\rm I}\right)
\end{equation}
$G(U,p_{\rm I})$ has the property that for every subgraph $G'\subseteq G(U,p_{\rm I})$
with
\[
    \big|E(G(U,p_{\rm I}))\setminus E(G')\big| \leq \frac{\delta_{\rm I}}{2}\binom{\rho n}{2}p_{\rm I}
\]
we have
\begin{equation}\label{eq:robRamsey}
    G'\xrightarrow{a_i\mu} F_i\,.
\end{equation}

Our  goal is to show that, with high probability, any two-coloring $\chi$ of
$G(U,p_{\rm I})$ yields at least $d(|U|^2/2)$ $\chi$-rich edges,  and ultimately, by repeating this
argument for every set $U\subseteq [n]$ with $\rho n$ vertices, that~$\Gamma_\chi$ is
$(\rho,d)$-dense.
 The above `robust'  Ramsey property
(\ref{eq:robRamsey})
 means that after applying Proposition~\ref{upper} to $G(U,p_{\rm I})$ the resulting subgraph of
$G(U,p_{\rm I})$ will still have the Ramsey property with high probability.

Let $Y$ be the random variable counting the number of double creatures in $G(U,p_{\rm I})$.
It follows from Fact \ref{T} that
\begin{equation}\label{eq:expY}
    \EE Y
    \le t(\rho n)^{2k-2}p_{\rm I}^{2i}\,.
\end{equation}
Hence, by Proposition~\ref{upper}, applied for every $j=1,\dots, t$
to the families $\mathcal S_j$ of all copies of $T_j$ in $G(U,p_{\rm I})$
with
\begin{equation}\label{eq:h1}
    h_{\rm I}=\frac{\delta_{\rm I}}{2t}\binom{\rho n}{2}p_{\rm I}
\end{equation}
we conclude that
with probability at least
\begin{equation}\label{eq:q1b}
    1-\sum_{j=1}^{t}\exp\left(-\frac{h_{\rm I}}{2e(T_j)}\right)
    \geq
    1-t\exp\left(-\frac{h_{\rm I}}{2k^2}\right)
\end{equation}
there exists a subgraph  $G_0\subseteq G(U,p_{\rm I})$ with $|E(G(U,p_{\rm I})\setminus E(G_0)|\leq
th_{\rm I}$ such that $G_0$ contains at most  $2\EE Y$  double creatures. Since
$$
    th_{\rm I}
    \overset{\eqref{eq:h1}}{=}
    \frac{\delta_{\rm I}}{2}\binom{\rho n}{2}p_{\rm I},
$$
the robust Ramsey property~\eqref{eq:robRamsey} holds with $G'=G_0$.

Recall that a two-coloring $\chi$ of $G(n,p_{\rm I})$ is fixed. For $\{u,v\}\subset U$, let
$x_{uv}$ be the number of monochromatic copies of $F_{i}$ in $G_0$ which together with the pair
$\{u,v\}$ form a copy of~$F_{i+1}$. Owing to~\eqref{eq:robRamsey}, we have
\begin{equation}\label{eq:xuv}
    \sum_{\{u,v\}\in\binom{U}{2}}x_{uv}\geq a_i\mu.
\end{equation}
By  the above application of Proposition~\ref{upper}  we infer that
\begin{equation}\label{eq:xuv2}
    \sum_{\{u,v\}\in\binom{U}{2}}x^2_{uv}\leq 2\cdot\binom{k}{2}\cdot|DC(G_0)|\leq4\binom{k}{2}\EE Y
    \overset{\eqref{eq:expY},\eqref{eq:t}}{\leq}2^{2k^2-1}(\rho n)^{2k-2}p_{\rm I}^{2i}\,,
\end{equation}
where $DC(G_0)$ is the set of all double creatures in $G_0$. Recall that $\{u,v\}\in
E(\Gamma_{\chi})$ if it is $\chi$-rich, which is implied by $x_{uv}\geq \l$, where $\ell$ is
defined in ~\eqref{eq:ell}. We want to show that 
\[
	e(\Gamma_{\chi}[U])\geq \frac{d}{2}({\rho n})^2\,.
\] 
Since
$\l\le a_i\mu/(\rho n)^2$ (compare~\eqref{eq:ell} and~\eqref{eq:mui}), it follows
from~\eqref{eq:xuv} that
\[
\sum_{\substack{\{u,v\}\in\binom{U}{2}\\ x_{uv}\geq \l}}x_{uv}\ge  \frac{a_i\mu}{2}
\overset{\eqref{eq:mui}}{\geq}
    \frac{1}{2}\cdot\frac{a_i}{4^{k^2}}(\rho n)^{k}p_{\rm I}^i\,.
\]
Squaring the last inequality and applying the Cauchy-Schwarz inequality yields
\begin{align*}
    \left(\frac{1}{2}\cdot\frac{a_i}{4^{k^2}}(\rho n)^{k}p_{\rm I}^i\right)^2
    \le
    \Bigg(\sum_{\substack{\{u,v\}\in\binom{U}{2}\\ x_{uv}\geq \l}}x_{uv}\Bigg)^2
    &\overset{\phantom{\eqref{eq:xuv2}}}{\leq}
    e(\Gamma_\chi[U]) \sum_{\substack{\{u,v\}\in\binom{U}{2}\\ x_{uv}\geq \l}}x^2_{uv}\\
    &\overset{\eqref{eq:xuv2}}{\leq}
    e(\Gamma_\chi[U])\cdot 2^{2k^2-1}(\rho n)^{2k-2}p_{\rm I}^{2i}\,.
\end{align*}
Consequently,
\[
    e(\Gamma_\chi[U]) \geq \frac{a^2_i}{64^{k^2}}({\rho n})^2/2
    \geq
    \frac{a^2_i}{64^{k^2}}({\rho n})^2/2\overset{\eqref{eq:dgd}}{=}d({\rho n})^2/2\,.
\]
Summarizing the above, we have shown that if $G(U,p_{\rm I})$ has the robust Ramsey property for
$F_i$ (see~\eqref{eq:robRamsey}) and if  the conclusion of Proposition~\ref{upper} holds for all $j=1,\dots,t$, then
we infer $e(\Gamma_{\chi}[U])\geq d({\rho n})^2/2$. The  probability that at least one of these events fails
is at most (see~\eqref{eq:q1a} and~\eqref{eq:q1b})
\[
\exp\left(-\frac{\delta_{\rm I}^2}{9}\binom{\rho n}{2}p_{\rm I}\right) + t\exp\left(-\frac{h_{\rm
I}}{2k^2}\right)\,. \]
 Recalling that $t\leq 4^{k^2}$ (see \eqref{eq:t}) and the definition of
$h_{\rm I}$ in~\eqref{eq:h1}, Claim~\ref{Gamma} now follows by summing up these probabilities over
all choices of $U\subseteq [n]$ with $|U|= \rho n$. More precisely, using the union bound and the
estimate $\binom n{\rho n}\le 2^n$, we conclude that the probability that there is a coloring
$\chi$ for which the graph $\Gamma_\chi$ is not $(\rho, d)$-dense is

\begin{align*}
\PP(\neg\mathcal{E})
	&\le2^{n}
	\exp\left(-\frac1{9}\delta_{\rm I}^2\binom{\rho n}2p_{\rm I}\right)
         +2^n4^{k^2}\exp\left(-\frac{1}{k^24^{k^2}}\delta_{\rm I}\binom{\rho n}2p_{\rm I}\right)
	\\
    &\leq
    \exp\left(-\frac{\delta_{\rm I}^2}{16^{k^2}}\binom{\rho n}{2}p_{\rm I}+n+2k^2\right)\qedhere
\end{align*}
\end{proof}

This concludes the analysis of the first round.

\subsubsection*{Second round.}
Let $\mathcal B$ be the conjunction of $\mathcal{E}$ and the event that $|G(n,p_{\rm I})|\le
n^2p_{\rm I}$. In the second round we will condition on the event $\mathcal B$ and sum over all
two-colorings $\chi$ of $G(n,p_{\rm I})$. Formally, let $\mathcal A$ be the (bad) event that there
is a two-coloring of the edges of $G(n,p)$ with fewer than $a_{i+1}\mu_{F_{i+1}}$ monochromatic
copies of $F_{i+1}$. (That is, $\neg\mathcal A$ is the Ramsey property $G(n,p)\xrightarrow{a_{i+1}\mu_{F_{i+1}}} F_{i+1}$.) Further, given a two-coloring $\chi$ of $G(n,p_{\rm I})$, let
$\mathcal{A}_\chi$ be the event that there exists an extension of $\chi$ to a coloring $\bar \chi$
of $G(n,p)$ yielding altogether fewer than $a_{i+1}\mu_{F_{i+1}}$ monochromatic copies of
$F_{i+1}$.

The following pair of inequalities exhibit the skeleton of our proof of Theorem \ref{rg}:

\begin{equation}\label{ske1}
\PP(\mathcal A)\le \PP(\neg\mathcal{B})+\sum_{G\in\mathcal{B}} \PP\Big(\mathcal{A}|G(n,p_{\rm
I})=G\Big)\PP\Big(G(n,p_{\rm I})=G\Big)
\end{equation}
 and
\begin{equation}\label{ske2}\PP(\mathcal{A}|G(n,p_{\rm I})=G)=
\PP\left(\bigcup_{\chi}\mathcal{A}_\chi\Big|G(n,p_{\rm I})=G\right) \le
2^{n^2p_1}\max_\chi\PP(\mathcal{A}_\chi|G(n,p_{\rm I})=G).
\end{equation}
By Claim \ref{Gamma} and Chernoff's
inequality (see, e.g.,~\cite{JLR}*{ineq.~(2.5)})
\begin{multline}\label{eq:R1}
  \PP(\neg\mathcal{B})\le \PP(\neg\mathcal{E})+\PP\Big(|G(n,p_{\rm I})|>n^2p_{\rm
I}\Big)\\
\le \exp\left(-\frac{\delta_{\rm I}^2}{16^{k^2}}\binom{\rho n}{2}p_{\rm I}+n+2k^2\right)+
  \exp\left(-\frac{1}{3}\binom{n}{2}p_{\rm I}\right)\\
    \leq
    \exp\left(-\frac{\delta_{\rm I}^2}{16^{k^2}}\binom{\rho n}{2}p_{\rm I}+n+2k^2+1\right)=:\qI\,.
\end{multline}

To complete the proof of Theorem \ref{rg} it is thus crucial to find an upper  bound on
\newline $\PP(\mathcal{A}_\chi|G(n,p_{\rm I})=G)$ which substantially beats the factor $2^{n^2p_1}$.

\begin{claim}\label{IIround}
For every $G\in \mathcal{B}$ and every two-coloring $\chi$ of $G$,
$$\PP(\mathcal{A}_\chi|G(n,p_{\rm I})=G)\le \exp\left(-\frac{\delta_{\rm II}^2\gamma}{9} n^2 p_{\rm II}\right)\,.$$
\end{claim}

The edges of $\Gamma_{\chi}$ are naturally two-colored according to the majority color among the
monochromatic copies of $F_i$ attached to them. We color an edge of $\Gamma_\chi$ \emph{pink} if it
closes at least $\ell/2$ red copies of $F_i$ and we color it \emph{azure} otherwise. Subsequently,
we apply Corollary~\ref{mon} to $\Gamma_{\chi}$ for $F_{i+1}$ and $d$ (chosen in~\eqref{eq:dgd}).
Note that in~\eqref{eq:dgd} we chose $\rho$ to facilitate  such an application. Moreover, the
required lower bound on $n$ is equivalent to $\rho n\geq 1$ and this follows from~\eqref{eq:rhon}.
Hence, by Corollary~\ref{mon} and the choice of $\gamma$ in~\eqref{eq:dgd}, we may assume without
loss of generality, that there are at least $\gamma n^k/2$ pink  copies of $F_{i+1}$ in
$\Gamma_\chi$. In particular, all these copies of $F_{i+1}$ consist entirely of edges closing each
at least $\ell/2$ red copies of $F_i$ (from the first round).
 Let us denote by ${\mathcal F}_\chi$ the family of these copies of~$F_{i+1}$, and let
 $\Gamma_\chi^{\rm pink}$ be the subgraph of $\Gamma_\chi$ containing the pink edges. Since
 every edge may belong to at most $n^{k-2}$ copies of $F_{i+1}$, we have
 \begin{equation}\label{eq:pink}
    e(\Gamma_\chi^{\rm pink})\ge\frac{(i+1)\cdot |{\mathcal F}_\chi|}{n^{k-2}}\ge\frac{(i+1)\cdot \gamma n^k/2}{n^{k-2}}
    \ge\gamma n^2\,.
 \end{equation}

 In the proof of Claim \ref{IIround} we intend to use again
 Proposition \ref{del}, this time with $\Gamma=\Gamma_\chi^{\rm pink}$ and
 $\mathcal Q$ -- the property of containing  at least
 \begin{equation}\label{eq:Q2}
    \frac{\gamma}{2^{k^2}} n^kp_{\rm II}^{i+1}
 \end{equation}
 copies of $F_{i+1}$ belonging to $\mathcal{F}_{\chi}$. For this, however, we need the following fact.

 \begin{fact}\label{Janson} With $\delta_{\rm II}$ chosen in~\eqref{eq:dgd} we have
$$
    \PP\big((\Gamma_\chi^{\rm pink})_{(1-\delta_{\rm II})p_{\rm II}}\not\in\mathcal Q\big)
    \le \exp\left(-\frac{\gamma^2}{4^{k^2}}e(\Gamma_\chi^{\rm pink})p_{\rm II}\right)\,.$$
 \end{fact}
 \begin{proof} Consider a random variable $Z$ counting the number
 of copies $F_{i+1}$ belonging to $\mathcal{F}_{\chi}$
 which are subgraphs of $G(n,(1-\delta_{\rm II})p_{\rm II})$. We have
 \begin{equation}\label{eq:expZ}
 \EE Z=|{\mathcal F}_\chi|((1-\delta_{\rm II})p_{\rm II})^{i+1}\ge\frac12\gamma n^k((1-\delta_{\rm II})p_{\rm II})^{i+1}
 \geq
 \frac{1}{2}\cdot\frac{1}{2^{\binom{k}{2}}}\gamma n^kp_{\rm II}^{i+1}
 \,,
 \end{equation}
 where we used the bound $\delta_{\rm II}\leq 1/2$.

By Janson's inequality (see, e.g., \cite{JLR}*{Theorem~2.14}),
 $$
    \PP\big((\Gamma_\chi^{\rm pink})_{(1-\delta_{\rm II})p_{\rm II}}\not\in\mathcal Q\big)
    \le
    \Pr\left(Z\le\frac12\E Z\right)
    \le
    \exp\left(-\frac{(\E Z)^2}{8\bar\Delta}\right),$$
where $\bar\Delta$ is defined by 
$$
\bar\Delta
=
\sum_{F'\in \mathcal F_{\chi}}\sum_{F''\in\mathcal F_\chi}
\Pr(F'\cup F''\subseteq G(n,(1-\delta_{\rm II})p_{\rm II}))\,,
$$
with the double sum ranging over all pairs  $(F',F'')\in\mathcal F_\chi\times\mathcal F_\chi$ such that $E(F')\cap E(F'')\neq\emptyset$.
The quantity $\bar\Delta$ can be bounded from above by
\begin{equation}\label{eq:bDelta}
    \bar\Delta\leq \sum_{\tF\subseteq F_{i+1}}n^{2k-v(\tF)}p_{\rm II}^{2(i+1)-e(\tF)}\,,
\end{equation}
where the sum is taken over all  subgraphs $\tF$ of $F_{i+1}$ with at least one edge. If
$e(\tF)=1$  then
\begin{equation}\label{eq:J1}
    n^{v(\tF)}p_{\rm II}^{e(\tF)}= n^{v(\tF)}p_{\rm II}\geq n^2p_{\rm II}\,.
\end{equation}
Otherwise,
\begin{equation}\label{eq:J2}
   n^{v(\tF)}p_{\rm II}^{e(\tF)}
    \geq
    \frac{n^{v(\tF)}p^{e(\tF)}}{2^{e(\tF)}}
  \overset{\eqref{npCH}}{\ge}
    \frac{n^2pC_{i+1}^{e(\tF)-1}}{2^{e(\tF)}}
    \geq
     n^2p
    \overset{\eqref{eq:ps}}{\ge}
    n^2p_{\rm II}\,,
\end{equation}
where we also used the fact that $C_{i+1}\geq 4$ (see~\eqref{rec}). Combining~\eqref{eq:bDelta} with the bounds~\eqref{eq:J1} and \eqref{eq:J2} yields
\[
    \bar\Delta
    \leq 2^{i+1}n^{2k-2}p_{\rm II}^{2i+1}
    \leq 2^{\binom{k}{2}}n^{2k-2}p_{\rm II}^{2i+1}\,.
\]
Finally, plugging this estimate for $\bar\Delta$ and~\eqref{eq:expZ} into Janson's inequality we
obtain
$$
    \Pr\big((\Gamma_\chi^{\rm pink})_{(1-\delta_{\rm II})p_{\rm II}}\not\in\mathcal Q\big)
    \le
    \exp\left(-\frac{\gamma^2n^2 p_{\rm II}}{32\cdot2^{2\binom{k}{2}}\cdot2^{\binom{k}{2}}}\right)
    \le
    \exp\left(-\frac{\gamma^2}{4^{k^2}}e(\Gamma_\chi^{\rm pink})p_{\rm II}\right)\,.
$$
\end{proof}

\begin{proof}[Proof of Claim \ref{IIround}]
We shall apply Proposition~\ref{del} with $c=\gamma^2/4^{k^2}$, $\delta_{\rm
II}=\gamma^4/(9\cdot 16^{k^2})$ (see~\eqref{eq:dgd}),  $N=e(\Gamma_\chi^{\rm pink})$, and
$p_{\rm II}$. Therefore, first we have to verify that $e(\Gamma_\chi^{\rm pink})p_{\rm II}\geq
72/\delta_{\rm II}^2$. Indeed,
\[
    e(\Gamma_\chi^{\rm pink})\cdot p_{\rm II}
    \overset{(\ref{eq:ps},\,\ref{eq:pink})}{\geq} \gamma n^2 \cdot \frac{p}{2}
\overset{(\ref{npC})}{\geq}
    \frac{\gamma}{2}nC_{i+1}
    \overset{\eqref{rec}}{\ge}
    \frac{\gamma}{2}\cdot \frac{2^{122k^4}}{a_i^{37k^3}}
    \overset{\eqref{eq:dgdr}}{\ge}
    \frac{72\cdot 81\cdot 16^{2k^2}}{\gamma^8}=\frac{72}{\delta_{\rm II}^2}\,.
\]
Consequently, by Proposition \ref{del}, we conclude that with probability at least
\begin{equation}\label{eq:error2}
1-\exp\left(- \frac{\delta_{\rm II}^2}{9} e(\Gamma_\chi^{\rm pink})p_{\rm II} \right)
\overset{\eqref{eq:pink}}{\ge}
1-\exp\left(-\frac{\delta_{\rm II}^2\gamma}{9} n^2 p_{\rm II}\right)\,,
\end{equation}
the random graph $(\Gamma_\chi^{\rm pink})_{p_{\rm II}}$ has the property that for every subgraph
$\Gamma'\subseteq (\Gamma_\chi^{\rm pink})_{p_{\rm II}}$
with
\begin{equation}\label{eq:h2}
    \big|E((\Gamma_\chi^{\rm pink})_{p_{\rm II}})\setminus E(\Gamma')\big|
    \leq
    \frac{\delta_{\rm II}\gamma}{2} n^2p_{\rm II}=:h_{\rm II}
\end{equation}
we have $\Gamma'\in\cQ$, that is, $\Gamma'$ contains at least
    $\frac{\gamma}{2^{k^2}} n^kp_{\rm II}^{i+1}$
 copies of $F_{i+1}$ belonging to $\mathcal{F}_{\chi}$ (see~\eqref{eq:Q2}).

Consider now an extension $\bchi$ of the coloring $\chi$ from $G(n,p_{\rm I})$ to $G(n,p)$. If in
the coloring~$\bchi$ fewer than $h_{\rm II}$ edges of $(\Gamma_\chi^{\rm pink})_{p_{\rm II}}$ are
colored red, then, by the above consequence of Proposition \ref{del}, the blue part of
$(\Gamma_\chi^{\rm pink})_{p_{\rm II}}$ contains at least
\[
    \frac{\gamma}{2^{k^2}} n^kp_{\rm II}^{i+1}
    \overset{\eqref{eq:ps}}{\geq}
    \frac{\gamma}{4^{k^2}} n^kp^{i+1}
\]
copies of $F_{i+1}$. If, on the other hand, more than $h_{\rm II}$ edges of
$(\Gamma_\chi^{pink})_{p_{\rm II}}$ are colored red, then, by the definition of a pink edge, noting
that $i\le k^2/2$, at least
\begin{align*}
    h_{\rm II}\times \frac{\ell}{2}\times \frac1{i+1}
    &\overset{(\ref{eq:ell},\,\ref{eq:h2})}{\geq}
    \frac{\delta_{\rm II}\gamma}{2} n^2p_{\rm II}\times \frac{a_i}{4^{k^2}k^2}(\rho n)^{k-2}p_{\rm I}^i\\
    &\overset{\hspace{6pt}\eqref{eq:ps}\hspace{6pt}}{\geq}
    \frac{\delta_{\rm II}\gamma}{4} n^2p\times \frac{a_i\rho^k}{4^{k^2}k^2}\left(\frac\alpha2\right)^in^{k-2}p^i\\
    &\overset{\phantom{(\ref{eq:ell},\,\ref{eq:h2})}}{\geq}
    \frac{\delta_{\rm II}\gamma a_i\rho^k\alpha^{k^2/2}}{16^{k^2}}n^kp^{i+1}
\end{align*}
red copies of $F_{i+1}$ arise. Owing to~\eqref{eq:dgd},~\eqref{eq:dgdr}, and the choice of
$a_{i+1}$ in~\eqref{rec} we have
\[
    \frac{\gamma}{4^{k^2}}
    \overset{\eqref{eq:dgdr}}{\ge}
    \frac{a_i^{4k^2}}{2^{13k^4+2k^2}}
    \overset{\eqref{rec}}{\ge}
    a_{i+1}
\]
and
\[
    \frac{\delta_{\rm II}\gamma a_i\rho^k\alpha^{k^2/2}}{16^{k^2}}
    \overset{\eqref{eq:dgd}}{=}
    \frac{\gamma^5\rho^k\alpha^{k^2/2}}{9\cdot 2^{8k^2}} a_i
    \overset{(\ref{eq:dgdr})}{\ge}
    \frac{a_i^{18k^4+24k^2}}{2^{55k^6}}
    \overset{\eqref{rec}}{\ge}
    a_{i+1}\,.
\]
Therefore, we have shown that with  probability as in \eqref{eq:error2}, indeed any extension
$\bchi$ of $\chi$ yields at least
$$\min\left(\frac{\gamma}{4^{k^2}}\,,\,\frac{\delta_{\rm II}\gamma a_i\rho^k\alpha^{k^2/2}}
{16^{k^2}}n^kp^{i+1}\right)\ge a_{i+1}n^kp^{i+1}\overset{\eqref{miu}}{\ge} a_{i+1}\mu_{F_{i+1}}$$
monochromatic copies of $F_{i+1}$.
\end{proof}

\subsubsection*{The final touch.}
To finish the proof of Theorem \ref{rg} it is left to verify that indeed 
\[
	\PP(\mathcal A)\le\exp(-b_{i+1}\tbinom n2p\,.
\] 
The error probability of the first round is (see \eqref{eq:R1})
$$\PP(\neg\mathcal B)\le q_{\rm I}.$$
Turning to the second round, by Claim \ref{IIround} and \eqref{ske2}, for any $G\in\mathcal B$,

\begin{equation}\label{eq:R2}
   \PP(\mathcal{A}|G(n,p_{\rm I})=G)\le 2^{n^2p_{\rm I}}\cdot\exp\left(-\frac{\delta_{\rm II}^2\gamma}{9} n^2 p_{\rm II}\right)
    \leq
    \exp\left(-\frac{\delta_{\rm II}^2\gamma}{9}n^2 p_{\rm II}+n^2 p_{\rm I}\right)=:\qII\,,
\end{equation}
and, consequently, by \eqref{ske1},
$$\PP(\mathcal A)\le\qI+\qII.$$
Below we show (see Fact \ref{qq}) that $\qI$ and $\qII$ are each upper bounded by
$\exp(-b_{i+1}n^2p)$. Consequently,
\[
    \PP(\mathcal A)\leq 2\exp(-b_{i+1}n^2p)\le\exp(1-b_{i+1}n^2p)\le\exp(-\tfrac{b_{i+1}}{2}n^2p)\le\exp(-b_{i+1}\tbinom n2p)
\]
because
\[
    \frac{b_{i+1}}{2}n^2p
    \overset{\eqref{npC}}{\geq}
    \frac{b_{i+1}}2C_{i+1}n_{i+1}
    \overset{\eqref{rec}}{\ge}
    C_in_{i+1}
    \ge
    1\,.
\]

\begin{fact}\label{qq}
We have $\max(\qI,\qII)\le\exp(-b_{i+1}n^2p)$.
\end{fact}
\begin{proof} We first bound $\qI$. Since $\rho n\ge3$ (see~\eqref{eq:rhon}),
\[
\frac{\delta_{\rm I}^2}{16^{k^2}}\binom{\rho n}2p_{\rm I}\overset{\eqref{eq:ps}}{\ge}
    \frac{\delta_{\rm I}^2\rho^2\alpha}{16^{k^2}\cdot 6}n^2p
    \overset{(\ref{eq:dgdr}\,,\ref{eq:delta_I})}{\geq}
    \frac{b_i^4a_i^{36k^2+8k}}{6^5\cdot2^{109k^4+26k^3+4k^2}}np^2
    \overset{\eqref{rec}}{\ge}
    2b_{i+1}n^2p
\]
while, since $i+1\geq 2$,
\[
    n+2k^2+1\leq n+n_{i+1}\leq 2n
    \overset{\eqref{rec}}{\leq}
    b_{i+1}C_{i+1} n\overset{\eqref{npC}}{\leq} b_{i+1}n^2p\,.
\]
Consequently,
\[
    \qI\leq \exp(-2b_{i+1}n^2p+b_{i+1}n^2p)=\exp(-b_{i+1}n^2p)\,.
\]
Now we derive the same upper bound for $\qII$. Since
\[
    p_{\rm I}
    \overset{\eqref{eq:ps}}\leq\alpha p
    \overset{\eqref{eq:dgd}}=\frac{\delta_{\rm II}^2\gamma}{36} p
\]
while $p_{\rm II}\overset{\eqref{eq:ps}}\geq p/2$,
\[
    \qII
    =
    \exp\left(-\frac{\delta_{\rm II}^2\gamma}{9}n^2 p_{\rm II}+n^2 p_{\rm I}\right)
    \leq
    \exp\left(-\frac{\delta_{\rm II}^2\gamma}{36}n^2p\right)\,.
\]
Therefore, the required bound follows from
\[
    \frac{\delta_{\rm II}^2\gamma}{36}
    \overset{\eqref{eq:dgd}}{=}
    \frac{\gamma^9}{36\cdot 81\cdot 16^{2k^2}}
    \overset{\eqref{eq:dgdr}}{\ge}
    \frac{a_i^{36k^2}}{2^{118^{k^4}}}
    \overset{\eqref{rec}}{\ge}
    b_{i+1}\,.\qedhere
\]
\end{proof}
This concludes the proof of the inductive step, i.e., the proof of Theorem~\ref{rg} for $F_{i+1}$, given it is true for $F_i$, $i=1,\dots,\binom k2-1$.
The proof of Theorem \ref{rg} is thus completed.

\section{Proof of Corolary \ref{main}}\label{proofcor}

In order to deduce Corollary \ref{main} from Theorem \ref{rg}, we first need to estimate the involved parameters
$a_i,b_i,C_i,n_i$, $i=1,\dots,\tbinom k2$, defined  recursively in (\ref{rec}).

\begin{prop}\label{e^e}
There exist positive constants $c_1,c_2,c_3,c_4>0$ such that for every $k\ge3$
$$a_{K_k}\ge 2^{-k^{(c_1\cdot k^2)}}\qquad b_{K_k}\ge 2^{-k^{(c_2\cdot k^2)}}\qquad C_{K_k}\le
2^{k^{(c_3\cdot k^2)}}\qquad n_{K_k}\le 2^{k^{(c_4\cdot k^2)}}.$$
\end{prop}

\begin{proof}  Throughout the proof we assume that $k\ge k_0$ for some sufficiently large constant~$k_0$.
Let  $x=19k^4$, $y=55k^6$, and set $\alpha_i=\log a_i$, $i=1,\dots,\binom k2$. Recall that
$a_1=\tfrac12$. The recurrence relation (\ref{rec}) becomes now
$$\alpha_i=x\alpha_{i-1}-y,$$
whose solution can be easily found as
$$\alpha_i=-x^{i-1}-y\frac{x^{i-1}-1}{x-1}$$
(note that $\alpha_1=-1$). Hence, for all $i=1,\dots,\binom k2$, and some constant $c_1>0$,
\begin{equation}\label{-a}
-\alpha_i=x^{i-1}+y\frac{x^{i-1}-1}{x-1}\le k^{c_1\cdot i}.
\end{equation}
 In particular,
$$a_{\binom k2}\ge 2^{-k^{c_1\cdot \binom k2}}\ge 2^{-k^{(c_1\cdot
k^2)}}.$$
 The recurrence relation for the $b_i$'s is more complex. With
$u=37k^2$ and $v=118k^4$, it reads  as
$$b_i=b_{i-1}^4a_{i-1}^u2^{-v}.$$
Thus, recalling that $b_1=\tfrac18$,
$$b_i8^{4^{i-1}}=\prod_{j=2}^i\left(\frac{b_j}{b_{j-1}^4}\right)^{4^{i-j}}=\prod_{j=2}^i\left(
a_j^{u}2^{-v}\right)^{4^{i-j}}.$$ Setting, $\beta_i=\log b_i$, and taking logarithms of both sides
and using \eqref{-a} we obtain, for some constant $c_2>0$,
\begin{align}
-\beta_i=3\cdot4^{i-1}+\sum_{j=2}^i4^{i-j}\left(u(-\alpha_j)+v\right)
&\le4^i+(i-1)4^{i-2}\left(u(-\alpha_i)+v\right)\nonumber\\
&\le4^i\left[1+i\left(uk^{(c_1\cdot
i)}+v\right)\right]\le k^{c_2\cdot i},\label{-b}
\end{align}
 where in the last step above we used estimates $4^i\le k^{2i}$ and $i\le k^2$.
In particular,
$$b_{\binom k2}\ge 2^{-k^{(c_2\cdot k^2)}}.$$

The recurrence relation for $C_i$ involves not only $C_{i-1}$ and $a_{i-1}$ but also $b_{i-1}$.
Nevertheless, its solution follows the steps of that for $b_i$. Indeed, we have
$$\frac{C_i}{C_{i-1}}=\frac{2^z}{b_{i-1}^4a_{i-1}^w},$$
where $z=122k^4$ and $w=37k^2$. Recalling that $C_1=1$,
$$C_i=\prod_{j=2}^i\frac{C_j}{C_{j-1}}=\prod_{j=2}^i\frac{2^z}{b_{j-1}^4a_{j-1}^w}$$
and, consequently, by \eqref{-a} and \eqref{-b}, for some constant $c_3>0$,
\begin{align*}
	\log C_i
	&\le
	(i-1)z+\sum_{j=2}^i\left(4(-\beta_j)+w(-\alpha_j)\right)\\
	&\le
	(i-1)\left(z+4(-\beta_i)+w(-\alpha_i)\right)
	\le k^2\left(z+4k^{(c_2\cdot i)}+wk^{(c_1\cdot
i)}\right)
	\le 
	k^{c_3\cdot i}.
\end{align*}
In particular,
$$C_{\binom k2}\le 2^{k^{(c_3\cdot k^2)}}.$$
Similarly, for some constant $c_4>0$,
$$n_i=\prod_{j=2}^i\frac{n_j}{n_{j-1}}=\prod_{j=2}^i\frac{2^{14k^3}}{a_{j-1}^{4k}}\le 2^{k^{(c_4\cdot i)}}$$
and, consequently,
\[n_{\binom k2}\le 2^{k^{(c_4\cdot k^2)}}\,.\qedhere\]
\end{proof}

We are going to prove Corollary \ref{main} by  the
probabilistic method. We will  show that for some $c>0$,  every $n\ge 2^{k^{c\cdot k^2}}$, and  a suitable
function  $p=p(n)$, with positive probability, $G(n,p)$  has simultaneously two properties:
$G(n,p)\rightarrow K_k$ and $G(n,p)\not\supset K_{k+1}$. The following simple lower  bound on
$\PP(G(n,p)\not\supset K_{k+1})$ has been already proved in \cite{folk} (see Lemma~3 therein).
For the sake of completeness we reproduce that short proof here.

\begin{lemma}\label{fkg}
For all $k,n\ge3$ and $C>0$, if $p=Cn^{-2/(k+1)}\le\tfrac12$
then
$$\PP(G(n,p)\not\supset K_{k+1})> \exp(-C^{\binom{k+1}2}n)\,.$$
\end{lemma}

\begin{proof} By applying the FKG inequality (see,
e.g., \cite{JLR}*{Theorem~2.12 and Corollary~2.13},  we obtain the bound
$$\PP(G(n,p)\not\supset K_{k+1})\ge \left(1-p^{\binom{k+1}2}\right)^{\binom n{k+1}}
    \ge
    \exp\left(-2C^{\binom{k+1}2} n^{-k}\tbinom n{k+1}\right)
    >
    \exp\left(-C^{\binom{k+1}2}n\right)\,,$$
where we used the inequalities $\binom n{k+1}< n^{k+1}/2$ and  $1-x\ge e^{-2x}$ for $0<x<\tfrac12$.
\end{proof}

Now, we are ready to complete the proof of Corollary~\ref{main}. For convenience, set $\bar b=b_{\binom k2}$, $\bar
C=C_{\binom k2}$, and $\bar n=n_{\binom k2}$.  Let $n\ge \bar n$ and $p=\bar C n^{-2/(k+1)}$. By
Theorem \ref{rg},
$$\PP(G(n,p)\rightarrow K_k )\ge 1-\exp\left\{-\bar b p\binom n2\right\}.$$
Let, in addition, $n\ge (2\bar C)^{(k+1)/2}$. Then, by Lemma  \ref{fkg},
$$\PP(G(n,p)\not\supset K_{k+1})> \exp\left\{-\bar b p\binom n2\right\}$$
and, in turn,
$$\PP(G(n,p)\rightarrow K_k \mbox{ and }G(n,p)\not\supset K_{k+1})>0.$$
Consequently, for every
$$n\ge n_0:=\max(\bar n,(2\bar C)^{(k+1)/2})$$
 there exists a  graph $G$ with $n$ vertices such that $G\to K_k$ but $G\not\supset K_{k+1}$.
 Finally, by Proposition \ref{e^e}, there exists $c>0$ such that $n_0\le 2^{k^{c\cdot k^2}}$.
 This way we have proved that $f(k)\le n_0\le 2^{k^{c\cdot k^2}}$.

\end{document}